\documentclass[11pt]{amsart}
\font\bbbld=msbm10 scaled\magstephalf

\newcommand{\bC}{\hbox{\bbbld C}}

\newcommand{\bR}{\hbox{\bbbld R}}

\newtheorem{theorem}{Theorem}[section]
\newtheorem{lemma}[theorem]{Lemma}
\newtheorem{proposition}[theorem]{Proposition}

 \theoremstyle{definition}
\newtheorem{definition}[theorem]{Definition}

\theoremstyle{remark}

\numberwithin{equation}{section}

\begin{document}
\setlength{\baselineskip}{1.2\baselineskip}

\title{H\"older continuous solutions to complex Hessian equations}

\author{Ngoc Cuong Nguyen}
\address{Institute of Mathematics, Jagiellonian University,
 \L ojasiewicza 6, 30-348 Krak\'ow, Poland. }
\email{Nguyen.Ngoc.Cuong@im.uj.edu.pl}

\begin{abstract}
We prove the H\"older continuity of the solution to complex Hessian equation
with the right hand side in $L^p$, $p>\frac{n}{m}$, $1< m< n$, in a
$m$-strongly pseudoconvex domain in $\bC^n$ under some additional conditions
on the density near the boundary and on the boundary data.

\end{abstract}

\maketitle

\bigskip

\section*{Introduction}
\label{Holder}
\setcounter{equation}{0}
Let $\Omega$ be an open bounded subset in $\bC^n$. For $1\leq m \leq n$, one considers
the Dirichlet problem with given $\phi \in C(\partial \Omega)$ and $f\in L^p(\Omega)$, $p>n/m$,
\begin{equation}
\label{heq}
\begin{cases}
	u \in SH_m \cap L^\infty(\Omega), \\
 	(dd^cu)^m\wedge \beta^{n-m}
	= f \beta^n & \text{ in } \;\; \Omega, \\
 	u= \phi & \text{ on } \;\; \partial \Omega,
\end{cases}
\end{equation}
where $SH_m(\Omega)$ is the set of $m$-subharmonic functions in $\Omega$,
$\beta = dd^c\|z\|^2$, and $d=\partial + \bar\partial$, $d^c = i(\bar \partial - \partial)$.
In the case
$m=1$ (resp. $m=n$) this equation is the Laplace equation for subharmonic
functions (resp. the complex Monge-Amp\`ere equation for plurisubharmonic functions).

The complex Monge-Amp\`ere equations have been investigated extensively over
last years. We refer the reader to \cite{Ce1}, \cite{K1}, \cite{EGZ}, \cite{PSS},
and references therein, for  accounts of recent results and more details.
We would like to emphasize here that the results on  H\"older continuity of  solutions
of complex Monge-Amp\`ere equations with the right hand side possibly degenerate
(see \cite{EGZ}, \cite{K3}, \cite{Demailly et al}) on compact K\"ahler manifolds
 turned out to be very useful in complex dynamic and complex geometry
(see e.g. \cite{DNS}, \cite{CDS}).

On the other hand, the complex Hessian equation is a rather new subject.  A
major progress has been done recently both for domains in $\bC^n$
(see \cite{Li}, \cite{Bl}, \cite{DK1}), and on compact K\"ahler manifolds,
(see \cite{HMW}, \cite{DK2}). In particular, the Calabi-Yau type theorem for complex
Hessian equations on a compact K\"ahler manifold was proved in \cite{DK2}.
It is expected to have some geometric applications, though not on the scale the  complex
Monge-Amp\`ere equations have.

The weak solution to complex Hessian equations have been studied
in \cite{Bl}, \cite{DK1}, \cite{Chi1, Chi2}, \cite{Cuong}.
It has been shown (see \cite{Bl},  \cite{DK1}, \cite{Chi1}) that
pluripotential theory can be adapted to $m$-subharmonic functions,
and it is a suitable tool for studying the weak solution to
complex Hessian equations with the right hand side in $L^p$, $p>n/m$.
Actually, Dinew and Ko\l odziej have obtained the continuous solution to the complex Hessian
equation for domains in $\bC^n$ (\cite{DK1}, Theorem 2.10) and for compact K\"ahler manifolds
(\cite{DK2}, Theorem 0.4).

In order to study the H\"older continuous solutions of the complex Hessian equation
on a general K\"ahler manifold it seems that the regularization techniques for
$\omega-$ $m$-subharmonic functions (\cite{DK1}, Definition 1.1)
 will play an important role (see \cite{K3}, \cite{Demailly et al}).  But in the case $1<m<n$,
 the problem of the regularization of non smooth $\omega-$ $m$-subharmonic functions
 for a general  K\"ahler form $\omega$  still needs to be solved.
Hence we restrict ourself to the case of domains in $\bC^n$ with the standard K\"ahler form
$\beta$.
Here we wish to study H\"older continuous solutions to \eqref{heq}
in a smoothly bounded, strongly $m$-pseudoconvex domain.

It is also motivated by the result in \cite{GKZ} for $m=n$,
where the equation \eqref{heq} becomes the complex Monge-Amp\`ere equation,
now considered in a strongly pseudoconvex domain.
Given $f\in L^p(\Omega)$, $p>1$, $\phi \in C(\partial \Omega)$ one seeks
$u$ such that
\begin{equation}
\label{ma}
\begin{cases}
	u \in PSH \cap L^\infty(\Omega), \\
 	(dd^cu)^n
	= f dV & \text{ in } \;\; \Omega, \\
 	u= \phi & \text{ on } \;\; \partial \Omega,
\end{cases}
\end{equation}
where $PSH(\Omega) \equiv SH_n(\Omega)$, $dV:=\beta^n$ is the Lebesgue measure.
It has been shown that (see \cite{K1})
the solution $u$ of \eqref{ma} is continuous.
Later on, in \cite{GKZ} the authors further showed that the
solution $u$ belongs to $Lip_\alpha(\bar \Omega)$,
$\alpha= \alpha (n, p)$, provided some additional assumptions on the
boundary data $\phi$ or on the Laplacian mass of $u$.

Our purpose is to prove the counterpart of the above result
for complex Hessian equations. More precisely, we want to show that
for $1<m<n$ the continuous solution $u$ to \eqref{heq} obtained in \cite{DK1}
is uniformly H\"older in $\Omega$, under some extra assumptions,
by using the potential theory developed in \cite{DK1} and suitable barrier arguments.
The main theorem is as follows.
\begin{theorem}
\label{mcor}
Let $\Omega$ be a smoothly bounded, strongly $m$-pseudoconvex domain, $1<m<n$.
Let $0\leq f\in L^p(\Omega), p>\frac{n}{m}$,  $\phi\in C^{1,1}(\partial\Omega)$,
and let $u$ be the solution of  \eqref{heq}.
\begin{enumerate}
 \item[(a)] If $f$ is bounded near the boundary $\partial \Omega$,
 	  then $u \in Lip_\alpha(\bar\Omega)$ for any
	  $0\leq \alpha < 2\gamma_1$;
\item[(b)] If $f(z)\leq C |\rho(z)|^{-m\nu}$ near $\partial \Omega$
	for some $C>0$, $0\leq \nu < \frac{1}{2}$, with $\rho$ being the
	defining function of $\Omega$ as in \eqref{rho},
	then $u\in Lip_\alpha(\bar \Omega)$ for any
	$0\leq \alpha< \gamma_2$,
\end{enumerate}
where $0< \gamma_1,\gamma_2 <\frac{1}{2}$  are uniform constants
defined in \eqref{se-gr}.
\end{theorem}

Recently, L.H. Chinh \cite{Chi2} also studied the H\"older continuity of the solution
to \eqref{heq} for $1<m<n$  by the  viscosity method.
In particular, in connection with our results in the case of a domain in
$\bC^n$,  he proved H\"older continuity of solutions in the strongly pseudoconvex domains
with the right hand side being at least continuous in $\bar \Omega$.
However, compared to Theorem~\ref{mcor},
he has put much less regularity on the boundary data
$\phi$, namely he took $\phi$ in a H\"older continuous class.

The organization of the paper is as follows, in Section~\ref{msub}  basic notions
related to $m$-subharmonic functions are recalled.
Section~\ref{stability} deals with stability estimates.
The crucial inequality is Proposition~\ref{se-dk} due to Dinew and Ko\l odziej,
which fills the gap for the case $1<m<n$ in order to get Theorem~\ref{se-th-5}.
In Section~\ref{he},  we first prove a more general statement in Theorem
~\ref{main}, and then we verify that under assumptions of
 Theorem~\ref{mcor} one can apply this statement.
 In particular, Theorem~\ref{hr} will show that any H\"older continuous function
 on the boundary can be extended to a $m$-sh  H\"older continuous function in the  whole domain.

\bigskip

\bigskip

\section*{Acknowledgements} I am grateful to my advisor, professor S\l awomir Ko\l odziej,
 who has patiently and thoroughly read the preliminary versions of this work.
His remarks and enlightening suggestions helped a lot to improve the exposition of the paper.
This work is supported by the International Ph.D Program
{\em" Geometry and Topology in Physical Models "}.

\bigskip

\section{$m$-subharmonic functions}
\label{msub}
\setcounter{equation}{0}
We briefly recall basic notions concerning $m$-subharmonic functions. We refer the reader to \cite{Bl},
\cite{DK1}  for a more detailed account. Let $\Omega$ be a bounded open subset in $\bC^n$.
Let $\beta := dd^c \|z\|^2$ denote the standard K\"ahler form in $\bC^n$,
where $d = \partial + \bar \partial $ and $d^c = i(\bar \partial - \partial)$.

\subsection{m-subharmonic functions}
\label{m-sub}
For $1 \leq m \leq n$ one considers the positive symmetric cone
\begin{equation}
\label{po-co}
	\Gamma_m = \{ \lambda \in \bR^n : \sigma_1(\lambda)> 0, ..., \sigma_m(\lambda) >0\},
\end{equation}
where $\sigma_{k}(\lambda ):= \sum_{1 \leq i_1 < ...< i_j\leq n} \lambda_{i_1} ... \lambda_{i_j} $ are
the $k$-th elementary symmetric polynomials of $\lambda$.  These symmetric cones
are convex (see \cite{Ga}).
\begin{definition}
\label{su-de-1}
Let $u$ be a subharmonic function in $\Omega$. \\
{\bf (a)} For smooth case, $u \in C^2(\Omega)$ is called $m$-subharmonic (m-sh for short)
if the eigenvalue values of the complex Hessian matrix form a vector
$$
	\lambda [\frac{\partial^2 u}{\partial z_j \partial \bar z_k} (z)] \in  \bar \Gamma_m,
\mbox{ equivalently }
	[dd^c u (z)]^k \wedge \beta^{n-k}(z) \geq 0, 1\leq k \leq m,  \;\; \forall z\in \Omega.
$$
{\bf (b)} For non-smooth case, $u$ is called $m$-sh if for any collection of
$v_1, ..., v_{m-1}$  $C^2$-smooth $m$-sh functions (in the definition (a)) the inequality
$$
	dd^c u \wedge dd^c v_1 \wedge ... \wedge dd^c v_{m-1} \wedge \beta^{n-m} \geq 0
$$
holds in the weak sense of currents in $\Omega$. \\
The set of all $m$-sh functions is denoted by $SH_m(\Omega)$.
\end{definition}
Following the Bedford and Taylor construction the wedge products of currents given by locally
bounded $m$-sh functions are well defined (defined inductively, see also \cite{Bl}).
\begin{proposition}
\label{su-pr-2}
Let $u_1$, .., $u_m$ be bounded $m$-sh functions then the measure
$$
	dd^c u_1 \wedge ... \wedge dd^c u_m  \wedge \beta^{n-m}
$$
is nonnegative.
\end{proposition}
It can be shown (see \cite{Bl}) that these nonnegative measures are continuous under monotone or
uniform convergence of their potentials.

\subsection{ $m$-pseudoconvex domains}
\label{mpsc}
Let $\Omega$ be a bounded domain with $\partial\Omega$ in the class $C^2$. Let $\rho\in C^2$
in a neighborhood of $\bar\Omega$ be a defining function of $\Omega$,
 i.e. a function such that
$$
	\rho<0 \;\; \text{on} \;\; \Omega, \;\;\;\;  \rho
		= 0 \;\; \text{and}\;\; d\rho\ne 0 \;\; \text{ on } \;\; \partial\Omega.
$$
\begin{definition}
\label{pr-de-4}
A $C^2$ bounded domain is called strongly $m$-pseudoconvex
if there is a defining function $\rho$ and some $\sigma>0$ such that
$(dd^c\rho)^k\wedge\beta^{n-k}\geq \sigma\beta^n$ in $\bar{\Omega}$
for every $1\leq k\leq m$.
\end{definition}

Using the defining function $\rho$ above together with the regularity of the boundary data
one can state  the following result for subharmonic
functions. This proposition seems to be  classical. Since we could not
find an accurate reference (see \cite{HanLin}, Lemma 1.35 and \cite{BT1},
Theorem 6.2 for example), we include its proof, which is based on \cite{BT1},
for the convenience of the reader.
This proposition will be used in the proof of Lemma~\ref{he-le-3}.
 \begin{proposition}
 \label{hh}
Let $\Omega$ be a smoothly bounded (i.e. strongly $1$-pseudoconvex) domain
and $\phi \in Lip_{2\alpha}(\bar \Omega)$, $0 < \alpha \leq \frac{1}{2}$.
Then the upper envelope
$$
	h(z) = \sup \{ v(z): v\in SH(\Omega)\cap C(\bar \Omega), \;\; v_{|_{\partial \Omega}} \leq \phi\}
$$
belongs to $Lip_\tau(\bar \Omega)$ for every $0< \tau \leq 2\alpha < 1$
(or for every $0 < \tau <1$ when $2\alpha = 1$).
 \end{proposition}

Here, and in the whole note, we use the notation:
$$
	Lip_\alpha(\bar \Omega) = \{ v \in C(\bar \Omega) : \|v\|_\alpha <+\infty\},
$$
  for $0<\alpha <1$, where the $\alpha$-H\"older norm  is given by
\begin{equation}
\label{hnorm}
	\|v\|_\alpha:= \sup\{ |v(z)| : z\in \bar \Omega\} +
		\sup_{z \neq w} \{\frac{|v(z) -v(w)|}{|z-w|^\alpha}: z,w \in \bar \Omega\}.
\end{equation}
It is also convenient if we consider in the case $\alpha = 0$ the space of continuous
functions in $\bar \Omega$,
and in the case $\alpha = 1$ the space of Lipschitz continuous functions with
uniform Lipschitz constants in $\bar \Omega$.
\begin{proof}
It is classical fact that
$h$ is a harmonic function in $\Omega$  with the boundary value $\phi$, and
it belongs to $C(\bar\Omega)$. In the next step
we will construct  subharmonic and superharmonic barriers at a given point on the boundary.
Let $\rho$ be the strictly subharmonic defining function of $\Omega$.
\begin{lemma}
\label{hh1}
Suppose that $\|\phi\|_{2\alpha} = M$, and
$0< \tau <1$ such that $ \tau \leq 2\alpha$.
Given $\xi \in \partial \Omega$
there is a uniform constant $K=K(\phi , \Omega) > 0$ such that the function
$$
a_{ \xi }(z) =  K |\rho|^\tau(z) + M |z-\xi |^{2 \alpha} + \phi (\xi)
$$
is superharmonic  in $\Omega \cap W$, where $W$ is
a neighborhood  of $\partial \Omega$. Moreover,
it is equal to $\phi (\xi)$ at $\xi$, and $a_\xi (z) \geq \phi (z)$  for every $ z\in \partial \Omega$.
\end{lemma}

\begin{proof}[Proof of Lemma~\ref{hh1}]
We have
$$
	dd^c (|\rho|^\tau) =
	       -	\tau |\rho|^{\tau -1} dd^c \rho - \tau (1-\tau) |\rho|^{\tau -2} d\rho \wedge d^c \rho \;
	       \mbox{ in } \Omega ,
$$
and
$$
	dd^c |z-\xi |^{2 \alpha} =  \alpha |z -\xi |^{2(\alpha -1)} dd^c |z -\xi |^2
		- \alpha (1- \alpha) |z-\xi|^{2(\alpha -2)} d |z-\xi |^2 \wedge d^c |z-\xi|^2.
$$
Hence,  we have, in $\Omega$,
$$
	dd^c a_\xi (z)  \wedge \beta^{n-1} (z)
		\leq -K \tau (1-\tau) |\rho|^{\tau -2} |\nabla\rho(z)|^2\beta^n(z)
			+M \alpha |z-\xi|^{2 (\alpha -1)} \beta^n(z).
$$
Furthermore, there exists $C>0$ such that
$$
	|\rho (z)| = | \rho(z) - \rho (\xi)| \leq C |z-\xi| \mbox{ for every } z\in \bar\Omega.
$$
Since $\tau -2 < 0$, it implies that, for $ z\in \Omega$,
\begin{equation}
\label{le3-1}
\begin{aligned}
	dd^c a_\xi (z)  \wedge \beta^{n-1}(z)
\; & \leq - 
		cst. K |z-\xi|^{\tau -2} |\nabla\rho(z)|^2\beta^n(z)
			+M \alpha |z-\xi|^{2 (\alpha -1)} \beta^n(z)   \\
\; & =  |z-\xi|^{2 (\alpha -1)} \left[
		M\alpha - 
				cst. K | z-\xi |^{\tau - 2\alpha} |\nabla \rho (z)|^2
							\right] \; \beta^n (z)\\
\;& \leq |z-\xi|^{2 (\alpha -1)} \left[
		M\alpha - 
				cst. K \; diam(\Omega)^{\tau - 2\alpha} |\nabla \rho (z)|^2
							\right] \; \beta^n (z) ,
\end{aligned}
\end{equation}
where $cst= C^{\tau -2 }  \tau (1-\tau)$, the last inequality follows from  the fact that
$\tau \leq 2 \alpha$ and  $ |z - \xi | \leq diam(\Omega) $.
As $\rho$ is a defining function of $\Omega$, $d\rho \neq 0$ on $\partial \Omega$,
one has $|\nabla \rho| >\varepsilon >0$
in a small neighborhood $W$ of $\partial \Omega$.
From \eqref{le3-1} we get that for $K>0$ big enough, independent of $\xi$,
$a_\xi (z)$ is superharmonic in $\Omega \cap W $.
The two latter properties follow from the formula for $a_\xi(z)$ and $M= \| \phi \|_{2\alpha}$.
\end{proof}

{\it End of Proof of Proposition~\ref{hh}.}
We may extend $a_\xi(z)$ to $\Omega$ as follows.
Let $U \subset \subset W$ be a neighborhood of $\partial \Omega$, and let
$0 \leq \chi \leq 1$
be a smooth cutoff function such that $\chi =1 $ in $\bar U$, and
$supp\, \chi \subset\subset W$. Since in $\Omega\setminus \bar U$ the function
$\chi (z) a_\xi(z)$ is smooth
the function
$\tilde a_\xi (z) = \chi (z) a_\xi (z)- K'\rho(z)$ is superharmonic in
$\Omega$  for $K'=K'(\Omega, \phi, \rho)>0$ big enough.
It is clear that $\tilde a_\xi \in Lip_\tau (\bar \Omega)$,
$\tilde a_\xi (z) \geq \phi(z)$ on $\partial \Omega$, and
$\tilde a_\xi(\xi) = \phi (\xi)$.
Finally, the superharmonic barrier is obtained by setting
$$
	a(z) := \inf \{ \tilde a_\xi (z) : \xi \in \partial \Omega \}.
$$
Then $- a(z)\in SH(\Omega) \cap Lip_\tau(\bar \Omega)$, and $a(z) = \phi(z)$
on $\partial \Omega$.
We have constructed a superharmonic function $a(z)$ in $\Omega$,
 and its boundary value is $\phi$.
Similarly, there is a subharmonic function
 $b \in SH(\Omega)\cap Lip_\tau (\bar \Omega)$  such that
 $b \leq h$ in $\Omega$, $b = \phi $ on $\partial \Omega$.
 According to the maximum principle, we have
 $$
 	b(z) \leq h(z) \leq a(z)
$$
in $\bar \Omega$. Therefore
\begin{equation}
\label{hh-1}
	|h(z) - h(\xi)| \leq K_1 |z-\xi|^\tau, \;\; K_1= K(a,b),
\end{equation}
for every  $z\in \bar \Omega$, $\xi \in \partial \Omega$. We will show that \eqref{hh-1} holds
for any $z, \xi \in \bar \Omega$.
For any small vector $w \in \bC^n$, define
$$
	V(z,w) =
		\begin{cases}
		   \max\{ h(z+w) - K_1|w|^\tau , h(z)\} \;\; & z, \; z+w \in \Omega ,\\
		   h(z) \;\; & z\in \bar \Omega, \;  z+w \notin \Omega.
		\end{cases}
$$
Observe that for all $ w$, the function $z\rightarrow V(z,w) \in SH(\Omega)$ by \eqref{hh-1},
and $V(\xi,w) = \phi(\xi)$ on
$\partial \Omega$. It follows that for all $z\in \Omega$, $V(z, w) \leq h(z)$. If $z+w \in \bar\Omega$,
this yields
$$
	h(z+w) - h(z) \leq K_1 |w|^\tau.
$$
Reversing the roles of $z+w$ and $z$, we obtain
$$
	|h(z+w) -h(z)| \leq K_1 |w|^\tau.
$$
Thus, $h \in Lip_\tau(\bar \Omega)$, and the proposition follows.
 \end{proof}

\subsection{Comparison principles}
\label{cp}
In next two sections,  we will need the following two comparison principles.
\begin{lemma}[Comparison principle]
\label{cp-1}
Let $\Omega$ be an open bounded subset in $\bC^n$.
For $u,v\in SH_m\cap  L^\infty(\Omega)$ satisfying
$\liminf_{\zeta\rightarrow z}(u-v)(\zeta)\geq 0$ for any $z\in \partial \Omega$, we have
$$
	\int_{\{u<v\}} (dd^cv)^m\wedge\beta^{n-m} \leq \int_{\{u<v\}} (dd^cu)^m\wedge\beta^{n-m}.
$$
Consequently, if $(dd^cu)^m\wedge\beta^{n-m}\leq (dd^cv)^m\wedge\beta^{n-m}$ in $\Omega$, then $v\leq u$ in $\Omega$.
\end{lemma}
\begin{proof}
See \cite{Cuong}, Theorem 1.14 and Corollary 1.15.
\end{proof}
\begin{lemma}
\label{cp-2}
Let $\Omega$ be an open bounded subset in $\bC^n$.
Let $u, v$ be continuous functions on $\bar \Omega$ and be $m$-sh functions in $ \Omega$.
Suppose that $u\leq v$ in $\Omega$ and $u=v$ on $\partial \Omega$. Then,
$$
	\int_\Omega dd^c v \wedge \beta^{n-1} \leq \int_\Omega dd^c u \wedge \beta^{n-1},
$$
$$
	\int_\Omega d v  \wedge d^c v \wedge \beta^{n-1}
		\leq \int_\Omega du\wedge d^c u \wedge \beta^{n-1}.
$$
\end{lemma}
\begin{proof}
The two inequalities are proved in the same way. We will only verify the first one.
 Set $v_\varepsilon = \max\{ v-\varepsilon , u\}$ for $\varepsilon >0$.
Since $u, v$ are continuous and $u=v$ on $\partial \Omega$,
one has $v_\varepsilon = u$ in a neighborhood of $\partial \Omega$. Hence
$$
	\int_\Omega dd^c v_\varepsilon  \wedge \beta^{n-1}
	= \int_\Omega dd^c u \wedge \beta^{n-1}.
$$
Moreover, $u\leq v$ in $\Omega$ it implies that $v_\varepsilon \nearrow v$ in $\Omega$.
Applying the monotone convergence theorem one obtains   weak* convergence
$dd^c v_\varepsilon \wedge \beta^{n-1} \rightarrow dd^c v \wedge \beta^{n-1}$. This implies
\begin{align*}
	\int_\Omega dd^c v \wedge \beta^{n-1}
\leq  \;& \liminf_{\varepsilon \rightarrow 0} \int_\Omega  dd^c v_\varepsilon \wedge \beta^{n-1}\\
= \;& \int_\Omega dd^c u \wedge \beta^{n-1} .
\end{align*}
The lemma follows.
\end{proof}

\bigskip

\section{Stability estimates}
\label{stability}
\setcounter{equation}{0}
In this section one considers $\Omega$ to be a bounded open set in $\bC^n$.
The main goal is to prove the stability estimate, Theorem~\ref{se-th-5},
in the case $1<m<n$.  The $m$-capacity, which is the version of the relative capacity
of plurisubharmonic functions
for $m$-sh functions, will play the analogous role in estimates as in the pluripotential case.
For $E$ a Borel set in $\Omega$ we define
$$
	cap_m(E,\Omega)
	= \sup \Big\{ \int_E (dd^cv)^m\wedge \beta^{n-m}: \;\; v\in SH_m(\Omega), 0\leq v\leq 1\Big\}.
$$
\begin{lemma}
\label{se-le-1}
Let $\varphi, \psi \in SH_m\cap L^\infty(\Omega)$ be such that
$\underline{lim}_{\zeta \rightarrow \partial\Omega} (\varphi - \psi)(\zeta) \geq 0$.
Then for all $t,s \geq 0$,
$$
	t^m cap_m (\{ \varphi - \psi <-s-t\},\Omega)
		 \leq \int_{\{\varphi - \psi< -s\}} (dd^c \varphi)^m \wedge \beta^{n-m}.
$$
\end{lemma}

\begin{proof}
Take $-1 \leq v \leq 0$ a $m$-sh function in $\Omega$.
Since $\{ \varphi +s < \psi -t\} \subset \{ \varphi+ s < \psi + tv\}$,
by the comparison principle (Lemma~\ref{cp-1}),
\begin{equation*}
\label{se-eq-1}
\begin{aligned}
	 t^m \int_{\{ \varphi - \psi < -s-t\}} (dd^c v)^m \wedge \beta^{n-m}
	&= \int_{\{ \varphi +s < \psi -t\}} [dd^c (tv)]^m \wedge \beta^{n-m} \\
	& \leq \int_{\{ \varphi +s < \psi +tv\}} (dd^c tv+ \psi)^m \wedge \beta^{n-m} \\
	& \leq \int_{\{ \varphi +s < \psi +tv\}} (dd^c \varphi)^m \wedge \beta^{n-m} \\
	& \leq \int_{\{ \varphi +s < \psi\}} (dd^c \varphi)^m \wedge \beta^{n-m},
\end{aligned}
\end{equation*}
where the last inequality used $\{ \varphi +s < \psi +tv\} \subset \{ \varphi +s < \psi \}$.
\end{proof}

The following result is an important inequality, due to
Dinew and Ko\l odziej (see \cite{DK1}, Proposition 2.1),
between the Euclidean volume and the $m$-capacity of Borel sets.
\begin{proposition}
\label{se-dk}
Let $\Omega$ be a bounded open subset in $\bC^n$ and
$0 \leq \tau < \frac{n}{n-m}$.
Then there exists a constant $C=C(\tau)>0$ such that for any Borel subset
$E \subset \subset \Omega$,
\begin{equation}
\label{dk-ineq}
V(E) \leq C [cap_m (E,\Omega)]^\tau,
\end{equation}
where $V:=\beta^n$ is the volume form.
\end{proposition}

It helps to obtain the following estimates in the case of $m$-subharmonic functions.

\begin{lemma}
\label{se-le-2}
Assume that $0 \leq f \in L^p(\Omega)$, $p > \frac{n}{m}$ and
$0< \alpha < \frac{p-\frac{n}{m}}{p(n-m)}$.
Then there exists a constant $C=C(\alpha,  \| f\|_{L^p(\Omega)})>0$ such that
$$
	\int_E f dV \leq C [cap_m (E, \Omega)]^{1+\alpha m}
$$
for every $E \subset \subset \Omega$.
\end{lemma}

\begin{proof}
Applying H\"older's inequality and then using \eqref{dk-ineq} with
$\tau = (1+ \alpha m)\frac{p}{p-1} < \frac{n}{n-m}$ we have
\begin{align*}
	 \int_{E} f dV
		& \leq \| f \|_{L^p(\Omega)} [V(E)]^\frac{1}{q}
			\leq C(\tau) \| f\|_{L^p(\Omega)} [cap_m(E, \Omega)]^\frac{\tau}{q}\\
		& \leq C (\alpha , \| f \|_{L^p( \Omega )})[cap_m(E, \Omega)]^{1+ \alpha m},
\end{align*}
where $\frac{1}{p}+\frac{1}{q} = 1$. Hence, the lemma follows.
\end{proof}

The following lemma was proved in \cite{EGZ}.

\begin{lemma}
\label{se-le-3}
Let $g: \bR^+ \rightarrow \bR^+$ be a decreasing right continuous function.
Assume there exists $\alpha> 0$ and $B>0$ such that
\begin{equation}
\label{se-le3-1}
	t g( s+t) \leq B [g(s)]^{1+ \alpha} \;\; \text{ for every } s, t \geq 0.
\end{equation}
Then $g(s) = 0$ for all $s \geq s_\infty$, where
$s_\infty:= \frac{2 B [g(0)]^{\alpha }}{1-2^{- \alpha}}$.
\end{lemma}

\begin{proof}
See Lemma 2.4 in \cite{EGZ}. The additional point is that the condition \eqref{se-le3-1}
holds for every
$s, t \geq 0$ while in \cite{EGZ}  the assumptions are for every $s\geq 0$
and for every $ 0\leq t \leq 1$.
Therefore, we may compute $s_\infty$ as in the statement.
\end{proof}

By combining Lemma~\ref{se-le-1}, Lemma~\ref{se-le-2} and Lemma~\ref{se-le-3}, we get

\begin{proposition}
\label{se-pr-4}
Let $\varphi, \psi \in SH_m\cap L^\infty (\Omega)$ be such that
$\underline{lim}_{\zeta \rightarrow \partial\Omega} (\varphi - \psi)(\zeta) \geq 0$,
and  $0\leq f\in L^p(\Omega)$, $p> \frac{n}{m}$.
Suppose that $(dd^c \varphi)^m \wedge \beta^{n-m} = f \beta^n$ in $\Omega$
and $0<\alpha<  \frac{p-\frac{n}{m}}{p(n-m)}$.
Then there exists a constant $A = A(\alpha , \| f \|_{L^p(\Omega)})$
such that for all $\varepsilon >0$,
$$ \sup_\Omega (\psi - \varphi) \leq \varepsilon + A [cap_m(\{\varphi -\psi <-\varepsilon \}, \Omega)]^\alpha.
$$
\end{proposition}

\begin{proof}
Put $g(s):= [cap_m(\{ \varphi- \psi < -\varepsilon -s \}, \Omega)]^{\frac{1}{m}}$.
Applying in turn Lemma~\ref{se-le-1}  and Lemma~\ref{se-le-2}, we obtain
\begin{align*}
t^m cap_m (\{ \varphi - \psi < -\varepsilon -s - t \}, \Omega)
	& \leq \int_{\{ \varphi - \psi < -\varepsilon -s\}} (dd^c \varphi)^m \wedge \beta^{n-m} \\
	&=  \int_{\{ \varphi - \psi < -\varepsilon -s\}} f \beta^n \\
	& \leq C(\alpha , \|f\|_{L^p(\Omega)}) [cap_m (\{ \varphi -\psi <-\varepsilon -s \}, \Omega)]^{1 + \alpha m}.
\end{align*}
Now, taking the m-th root of two sides one gets that
\begin{equation}
\label{se-pr4-1}
t g(s+t) \leq B [g(s)]^{1+\alpha m} \;\; \text { with } \;\; B = C^\frac{1}{m} (\alpha , \|f\|_{L^p(\Omega)}).
\end{equation}
From \eqref{se-pr4-1} we see that $g(s)$ satisfies assumptions of Lemma~\ref{se-le-3}.
It tells us that $[g(s_\infty)]^m = cap_m(\{ \varphi - \psi < - \varepsilon - s_\infty \}, \Omega) = 0$,
 which means
$ \psi - \varphi \leq \varepsilon + s_\infty$ almost everywhere.
Finally, by inserting into formula  $s_\infty = \frac{2B [g(0)]^{\alpha m}}{1 - 2^{-\alpha m}}$  we obtain
$$
	\sup_\Omega (\psi - \varphi) \leq \varepsilon +
		A \left [ cap_m(\{ \varphi -\psi <- \varepsilon \}, \Omega) \right ]^\alpha,
$$
where $A = \frac{2B}{1 - 2^{-\alpha m}}$.
\end{proof}

We are now in the position to prove the main stability estimate
which is similar to Theorem 1.1 in \cite{GKZ} for $m=n$ (see also \cite{DK1}, Theorem 2.5 in the case $1<m<n$).
In order to simplify the notation, from now on when $p>\frac{n}{m}$,
$1<m<n$, $\frac{1}{p}+\frac{1}{q} =1$ we set
\begin{equation}
\label{se-gr}
 \gamma_r := \frac{r}{ r+ mq +\frac{pq(n-m)}{p-\frac{n}{m}}}, \;\; \mbox{ for } r\geq 1.
\end{equation}

\begin{theorem}[Stability estimate]
\label{se-th-5}
Let $\varphi , \psi \in SH_m\cap L^\infty (\Omega)$ be such that
$\underline{lim}_{\zeta \rightarrow \partial\Omega} (\varphi - \psi)(\zeta) \geq 0$,
and
$0\leq f \in L^p(\Omega)$, $p> \frac{n}{m}$.
Suppose that $(dd^c\varphi)^m\wedge \beta^{n-m} = f\beta^n$ in $\Omega$.
Fix $r\geq 1$ and $0 < \gamma < \gamma_r$.
Then there exists a constant $C=C(\gamma, \|f\|_{L^p(\Omega)}) >0$ such that
$$ \sup_\Omega (\psi -\varphi) \leq C \left [\| (\psi - \varphi)_+\|_{L^r(\Omega)}\right ]^\gamma,
$$
where $(\psi -\varphi)_+ = \max \{\psi -\varphi , 0\}$.
\end{theorem}

\begin{proof}
We follow the lines of the proof of Theorem 1.1 in \cite{GKZ}.
Applying Lemma~\ref{se-le-1}  with $s=t = \varepsilon >0$ and then using H\"older's inequality, we get
\begin{align*}
cap_m(\{\varphi -\psi <-2\varepsilon \}, \Omega)
	& \leq \varepsilon^{-m} \int_{\{\varphi-\psi<-\varepsilon\}} f dV  \\
	& \leq \varepsilon^{-m-\frac{r}{q}} \int_{\Omega} (\psi -\varphi)_+^\frac{r}{q} f dV \\
	& \leq \varepsilon^{-m-\frac{r}{q}}\|f\|_{L^p(\Omega)} \|(\psi -\varphi)_+\|_{L^r(\Omega)}^\frac{r}{q}.
\end{align*}
Fix $0< \alpha < \frac{p-\frac{n}{m}}{p (n-m)}$ to be chosen later. Applying Proposition~\ref{se-pr-4} we have
\begin{equation}
\label{se-th5-1}
 \sup_\Omega ( \psi - \varphi)
		\leq 2 \varepsilon
		+ A \varepsilon^{- \alpha (m+\frac{r}{q})}
		\|f\|_{L^p(\Omega)}^\alpha \|(\psi - \varphi)_+\|_{L^r(\Omega)}^{\frac{\alpha r}{q}},
\end{equation}
where $A=A(\alpha , \| f \|_{L^p(\Omega)})$.
Now, we choose $\varepsilon = \|(\psi -\varphi_+)\|_{L^r(\Omega)}^\gamma$
and $\alpha = \frac{\gamma q}{r - \gamma (mq +r)}$ which is well defined,
the condition $0 < \gamma < \gamma_r$ being equivalent to
$0< \alpha < \frac{p -\frac{n}{m}}{p(n-m)}$.
Then the inequality \eqref{se-th5-1} becomes
$$
	\sup_\Omega ( \psi - \varphi)
		\leq A \|f\|_{L^p(\Omega)}^\alpha \|(\psi - \varphi)_+\|_{L^r(\Omega)}^\gamma .
$$
Thus, the theorem is proved.
\end{proof}

\bigskip

\section{H\"older continuity of the solution}
\label{he}
\setcounter{equation}{0}
Let $\Omega$ be a smooth bounded,
strongly $m$-pseudoconvex domain in $\bC^n$, $1< m < n$. We consider the Dirichlet
problem for the complex Hessian equation in the class of  $m$-sh functions.
\begin{equation}
\label{heq2}
\begin{cases}
 	(dd^cu)^m\wedge \beta^{n-m}
	= f \beta^n & \text{ in } \;\; \Omega, \\
 	u= \phi & \text{ on } \;\; \partial \Omega.
\end{cases}
\end{equation}

From the recent result of Dinew and Ko\l odziej (see \cite{DK1}, Theorem 2.10)
we know that $u\in SH_m(\Omega)\cap C(\bar \Omega)$ when
$f\in L^p(\Omega)$, $p>n/m$ and $\phi \in C(\partial \Omega)$.
After establishing the stability estimates in Section~\ref{stability},
we may use the scheme of the proof in \cite{GKZ}
in order to obtain further the H\"older continuity of the solution $u$ to \eqref{heq2}
under some additional assumptions.

\begin{theorem}
\label{main}
Let $0\leq f\in L^p(\Omega), p>\frac{n}{m}$, and $\phi\in C(\partial\Omega)$.
Let $u$ be the continuous solution to \eqref{heq2}.
Suppose that there exists $b\in Lip_\nu(\bar\Omega)$, $0<\nu<1$,
such that $b\leq u$ in $\Omega$, $b=u$ on $\partial\Omega$.
\begin{enumerate}
\item[(a)] If  $\nabla u$ belongs to $L^2(\Omega)$,
		then $u\in Lip_\alpha(\bar\Omega)$ for any
		$0 \leq \alpha<\min\{\nu,\gamma_2\}$.
\item[(b)] If the total mass of  $\Delta u$ is finite,
		then $u\in Lip_\alpha(\bar\Omega)$ for any
		 $0 \leq \alpha < \min\{\nu,2\gamma_1\}$. \label{main-2}
\end{enumerate}
Where $\gamma_1, \gamma_2$ are defined in \eqref{se-gr}.
\end{theorem}

It is not too difficult to see that when the total mass of $\Delta u$ is finite,
$\nabla u \in L^2(\Omega)$. However, the H\"older exponent in
Theorem~\ref{main}-(b) is better than the one in Theorem~\ref{main}-(a),
namely $\gamma_2 < 2 \gamma_1 $.
If we put some extra assumptions on the growth of the
density $f$ near the boundary and on the boundary data $\phi$,
 then we may verify the assumptions of Theorem~\ref{main}, which is the content of
 the main theorem (Theorem~\ref{mcor}).


\begin{proof}[ Proof of Theorem~\ref{main}]
For a fixed $\delta>0$, we set
\begin{equation*}
\label{he-omegadelta}
 	\Omega_\delta := \{ z\in \Omega : dist(z, \partial \Omega )>\delta \};
\end{equation*}
\begin{equation}
\label{he-maxu}
	u_\delta(z) :=\sup_{\|\zeta\|\leq \delta} u(z+\zeta),\;\; z\in \Omega_\delta;
\end{equation}
\begin{equation}
\label{he-mu}
 \hat u_\delta (z)
 	:= \frac{1}{v_{2n}\delta^{2n}} \int_{|\zeta -z|\leq \delta} u(\zeta) dV_{2n}(\zeta),
	 \;\; z \in \Omega_\delta,
\end{equation}
where $v_{2n}$ is the volume of the unit ball in $\bC^n$.
The following lemma shows that the H\"older norm  (see \eqref{hnorm}) of $u$  in $\bar\Omega$
can be computed by using either \eqref{he-maxu} or \eqref{he-mu}.

\begin{lemma}
\label{he-le-1}
Given $0<\alpha<1$, the following two conditions are equivalent.
\begin{enumerate}
 \item[(i)] There exists $\delta_1 , A_1> 0$ such that for any $0<\delta \leq \delta_1$,
$$
	u_\delta - u \leq A_1 \delta^\alpha \;\; \text{ on } \;\; \Omega_\delta.
$$
 \item[(ii)] There exists $\delta_2 , A_2>0$ such that for any $0<\delta \leq \delta_2$,
$$
	\hat u_\delta - u \leq A_2 \delta^\alpha \;\; \text{ on } \;\; \Omega_\delta.
$$
\end{enumerate}
\end{lemma}

\begin{proof}
See Lemma 4.2 in \cite{GKZ} where its proof used only the subharmonicity.
\end{proof}

The assumption $\nabla u \in L^2(\Omega)$ (resp. $\| \Delta u \| (\Omega) <+\infty$)
will  enable us to control the growth of
$\| u_\delta - u\|_{L^2(\Omega_\delta)}$ (resp. $\| \hat u_\delta - u\|_{L^1(\Omega_\delta)}$).

\begin{lemma}
\label{he-le-2}
For $\delta >0$ small enough, we have inequalities
\begin{equation}
\label{he-le2-1}
	\int_{\Omega_\delta} |u_\delta(z) - u(z)|^2dV_{2n}(z)
		\leq c_n \|\nabla u\|_{L^2(\Omega)}^2 \; \delta^2,
\end{equation}
\begin{equation}
\label{he-le2-2}
	 \int_{\Omega_\delta} [\hat u_\delta(z) - u(z)] dV_{2n}(z)
	 	\leq c_n \|\Delta u\|_{\Omega} \; \delta^2,
\end{equation}
where $c_n>0$ depends only on $n$.
\end{lemma}

\begin{proof}
See Lemma 4.3 and the last part in the proof of Theorem 3.1 in \cite{GKZ}.
There, only the subharmonicity was needed.
\end{proof}

In view of Lemma~\ref{he-le-1} and Lemma~\ref{he-le-2}
we wish to apply the stability estimate, Theorem~\ref{se-th-5},
to $\varphi:=u$ and $\psi:= u_\delta$.
This will give us the H\"older norm estimate in $\bar\Omega$ of the solution $u$,
in terms of $L^2$ norm of its gradient or its Laplacian mass in $\Omega$,
using \eqref{he-le2-1} or \eqref{he-le2-2}.
The remaining thing that we need is extending
$u_\delta$ to $\Omega$ (since it is only defined on $\Omega_\delta$),
in such a way that
after the extension the H\"older norm of $u$ is still under control.
It will be done with the help of the barrier function $b\in Lip_\nu(\bar\Omega)$,  $0 <\nu <1$.

\begin{lemma}
\label{he-le-3}
Under the assumptions of Theorem~\ref{main},
there exists a constant $c_0=c_0(b,\Omega)>0$ and $\delta_0$ small enough such that for any
$0<\delta< \delta_0$
\begin{equation}
\label{he-le3-1}
 	u_\delta(z) \leq u(z)+ c_0 \delta^{\nu} \;\; \text{for every} \;\; z\in \partial \Omega_\delta,
\end{equation}
where $\nu$ is the H\"older exponent of $b$. Consequently, the function
\begin{equation}
\label{he-le3-2}
 	\tilde u_\delta
 		= \begin{cases}
			\max\{u_\delta, u+ c_0 \delta^\nu\} & \text{ in } \;\; \Omega_\delta, \\
                                   u + c_0 \delta^\nu & \text{ in } \;\; \Omega\setminus \Omega_\delta
                     \end{cases}
\end{equation}
is a $m$-subharmonic function in $\Omega$, and it is continuous in $\bar\Omega$.
\end{lemma}

\begin{proof}
Let $h$ be the harmonic extension to $\Omega$ with $b$ as the boundary value on
$\partial\Omega$.  By Proposition~\ref{hh}
$h \in Lip_\nu (\bar \Omega)$. It is clear that $b \leq u \leq h$ in $\Omega$.
Fix a point $z\in \partial \Omega_\delta$,
there is $\zeta \in \bC^n$ with $\|\zeta\|=\delta$ such that
$z + \zeta \in \bar \Omega$ and $u_\delta(z) = u(z+\zeta)$. This yields
\begin{equation}
\label{he-le3-3}
\begin{aligned}
 u_\delta (z) - u(z)
 	= \;& u(z+\zeta) - u(z) \\
	\leq \;& h(z+\zeta)-u(z) \\
	\leq \;& h(z+\zeta) - b (z).
\end{aligned}
\end{equation}
Now, choose $\zeta_0 \in \bC^n$, $\|\zeta_0\| = \delta$, such that
$z+\zeta_0 \in \partial \Omega$.
It implies that $h(z+\zeta_0) = b(z+\zeta_0)$ as $b=u=h$ on $\partial \Omega$.
Then,
\begin{equation}
\label{he-le3-4}
\begin{aligned}
	h(z+\zeta) - b (z)
	= \;& [h(z+\zeta) -h(z)] + [h(z) - b(z)] \\
	\leq \;& \| h \|_\nu \; \delta^\nu + [h(z) - h(z+\zeta_0)] +[b(z +\zeta_0)-b(z)]\\
	\leq \;& c_0 \; \delta^\nu \;  \mbox{ with } c_0= 2 \|h\|_\nu +\| b\|_\nu .
\end{aligned}
\end{equation}
From \eqref{he-le3-3} and \eqref{he-le3-4} we obtain  \eqref{he-le3-1}.
Properties of $\tilde u_\delta$ in \eqref{he-le3-2} follow from the standard gluing procedure.
\end{proof}
\begin{proof}[Proof of $(a)$ in Theorem~\ref{main}]
Given $0< \gamma < \gamma_2$. Applying Theorem~\ref{se-th-5} with
$\varphi := u+c_0\delta^\nu$, $\psi := \tilde u_\delta$ and $r=2$ we get
$$
	\sup_\Omega [\tilde u_\delta-(u+c_0\delta^\nu)]
		\leq C\|(\tilde u_\delta - u-c_0\delta^\nu)_+\|_{L^2(\Omega)}^\gamma.
$$
Since $\tilde u_\delta = u+ c_0\delta^\nu$ in $\Omega\setminus \Omega_\delta$,
it implies that
$$
	\sup_{\Omega_\delta} (u_\delta -u - c_0\delta^\nu)
		\leq C \|(u_\delta -u - c_0\delta^\nu)_+\|_{L^2(\Omega_\delta)}^\gamma.
$$
As $(u_\delta -u - c_0\delta^\nu)_+ \leq u_\delta -u$ and using \eqref{he-le2-1}, we have
$$
	\sup_{\Omega_\delta} (u_\delta - u) \leq c_0\delta^\nu
		+ C \|u_\delta -u \|_{L^2(\Omega_\delta)}^\gamma
		 \leq c_0  \delta^\nu
		 	+ C.c_n^\frac{\gamma}{2}\|\nabla u\|_{L^2(\Omega)}^\gamma \; \delta^\gamma.
$$
Hence,
\begin{equation}
\label{he-pf-1}
	 \sup_{\Omega_\delta} (u_\delta -u)
	 	\leq c_1 \delta^{\min\{\nu,\gamma\}}
			\mbox{ for } \delta \mbox{ small enough},
\end{equation}
where $c_1= c_0 + C.c_n^\frac{\gamma}{2}\|\nabla u\|_{L^2(\Omega)}^\gamma$.
This finishes the first part of Theorem~\ref{main}
\end{proof}


\begin{proof}[Proof of $(b)$ in Theorem~\ref{main}]
Given $0< \gamma< \gamma_1$. The formula
 \eqref{he-le3-1} implies $\hat u_\delta \leq u_\delta\leq u + c_0\delta^\nu$ on
 $\partial\Omega_\delta$. Therefore, the function
$$
	u'_\delta
		= \begin{cases}
			 \max\{\hat u_\delta, u+ c_0 \delta^\nu\} & \text{ in } \;\; \Omega_\delta, \\
                     		u + c_0 \delta^\nu & \text{ in } \;\; \Omega\setminus \Omega_\delta
             	   \end{cases}
$$
is $m$-subharmonic in $\Omega$, and it is continuous in $\bar\Omega$.
Applying again Theorem~\ref{se-th-5} with
$\varphi := u+c_0\delta^\nu$, $\psi :=  u'_\delta$ and $r=1$ we get
$$
	\sup_\Omega [u'_\delta-(u+c_0\delta^\nu)]
		\leq C\|( u'_\delta - u-c_0\delta^\nu)_+\|_{L^1(\Omega)}^\gamma.
$$
Since $ u'_\delta = u+ c_0\delta^\nu$ in $\Omega\setminus \Omega_\delta$,
it follows that
$$
	\sup_{\Omega_\delta} (\hat u_\delta -u - c_0\delta^\nu)
	 	\leq C \|(\hat u_\delta -u - c_0\delta^\nu)_+\|_{L^1(\Omega_\delta)}^\gamma.
$$
Since $(\hat u_\delta -u - c_0\delta^\nu)_+ \leq \hat u_\delta -u$ and using \eqref{he-le2-2},
we get
$$
	\sup_{\Omega_\delta} (\hat u_\delta - u)
		 \leq c_0\delta^\nu + C \|\hat u_\delta -u \|_{L^1(\Omega_\delta)}^\gamma
		 \leq c_0\delta^\nu + C.c_n^\gamma\|\Delta u\|_{\Omega}^\gamma \; \delta^{2\gamma}.
$$
Hence,
$$
	\sup_{\Omega_\delta} (\hat u_\delta -u)
		\leq c_2 \delta^{\min\{\nu,2\gamma\}}
		\mbox{ for } \delta \mbox{ small enough},
$$
where  $c_2=  c_0 + C.c_n^\gamma\|\Delta u\|_{\Omega}^\gamma$.
Applying Lemma~\ref{he-le-1} one obtains
\begin{equation}
\label{he-pf-2}
	\sup_{\Omega_\delta} ( u_\delta -u)
		\leq c_3 \delta^{\min\{\nu,2\gamma\}}
		\mbox{ for } \delta \mbox{ small enough},
\end{equation}
for some uniform constant $c_3>0$.
This proves the second part.
\end{proof}
Thus, we have finished the proof of Theorem~\ref{main}.
\end{proof}

\bigskip


We proceed to prove Theorem~\ref{mcor}.
We fix a defining function $\rho$ of $\Omega$, setting
\begin{equation}
\label{rho}
\begin{aligned}
	&\Omega=\{\rho<0\}, \;\; \rho\in C^2(\bar \Omega), \\
	&(dd^c \rho)^k \wedge \beta^{n-k}\geq \sigma \beta^n \mbox{ on }\bar \Omega, \\
	&1\leq k \leq m, \;\;  \sigma > 0.
\end{aligned}
\end{equation}
We first prove the following two lemmas. The first one (cf. Proposition~\ref{hh})
was proved in \cite{D1} (see also \cite{BT1}) for $m=n$.

\begin{lemma}
\label{cor-le-1}
If $\phi\in C^{1,1}(\partial \Omega)$, then the upper envelope
$$
	h(z)= \sup\{v(z) : v\in SH_m(\Omega)\cap C(\bar\Omega), v\leq \phi \text{ on } \partial\Omega\}
$$
is a $m$-subharmonic function in $\Omega$
and is Lipschitz continuous in $\bar \Omega$.
It satisfies $h=\phi$  on $\partial\Omega$. Moreover,
\begin{equation}
\label{cor-le1-1}
	 \int_\Omega dd^ch\wedge \beta^{n-1} < +\infty.
\end{equation}
\end{lemma}

\begin{proof}
It is clear that $h \in SH_m(\Omega) \cap C(\bar \Omega)$,
and $h = \phi$ on $\partial \Omega$ because it coincides with
the unique continuous solution to \eqref{heq2}, where $f \equiv 0$.
We consider $\rho$ defined in $\eqref{rho}$.
There is an extension $\hat \phi$ of $\phi$ to a neighborhood $U \supset \bar \Omega$
such that
$\| \hat \phi\|_{C^{1,1}(\bar U)} \leq C \| \phi\|_{C^{1,1}(\partial \Omega)}$,
where $C= C(\Omega, U) >0$
(see \cite{GT}, Lemma 6.38).
Hence, for $A >0$ big enough, $A\rho \pm \hat \phi$ belongs to $SH_m(U)$.
Moreover, we can take $C$ so big that
\begin{equation}
\label{cor-le1-2}
	 \| A \rho \pm \hat \phi\|_{C^{1,1}(\bar U)}
		\leq C \left( 1+ \| \phi\|_{C^{1,1}(\partial \Omega)} \right).
\end{equation}
 The definition of $h$ implies
\begin{equation}
\label{cor-le1-3}
	A \rho + \hat\phi \leq h \leq \hat\phi - A \rho \text{ in } \Omega,
\end{equation}
where the second inequality follows from the maximum principle
for  the subharmonic function $h + A \rho - \hat\phi$ in $\Omega$.
We now extend $h$ to $\bar U$ by putting
\begin{equation}
\label{cor-le1-4}
\hat h(z) = \begin{cases}
			h(z) & \text{ for } z \in \Omega, \\
                            A \rho(z) + \hat\phi(z) & \text{ for } z \in \bar U\setminus \Omega.
                	\end{cases}
\end{equation}
According to \eqref{cor-le1-3} and \eqref{cor-le1-4},
$\hat h \leq \max\{\hat\phi - A\rho, \hat\phi + A\rho\}$ in $\bar U$.
For $\xi \in \partial\Omega$, $|w|$ so small that $\xi+w\in U$, we have
\begin{align*}
	\hat h(\xi+w)
		& \leq \phi( \xi ) + \max\{ \| \hat\phi - A \rho \|_{C^1(\bar U)}, \| \hat\phi
			+ A \rho \|_{C^1(\bar U)}\} |w| \\
            	& \leq \phi( \xi ) + C' |w|,
\end{align*}
where the last inequality follows from \eqref{cor-le1-2} with
$C' = C( 1+\| \phi \|_{C^{1,1}(\partial\Omega)} )$. It implies that
$$
	\hat h(\xi +w) - C'  |w|
		\leq \phi( \xi ) \text{ for every } \xi \in \partial\Omega.
$$
Hence, from the definition of $h$,
$\hat h( z +w) - C'  |w|\leq h(z)$ in $\bar\Omega$.
By changing $w$ into $-w$, we get, for $|w|$ so small that $z+w \in \bar\Omega$,
$$
	|h(z+w) - h(z)| \leq C'  |w| \text{ for } z\in \bar\Omega,
$$
since $\hat h(z)= h(z)$ in $\bar\Omega$.
Thus $\| h \|_{C^1(\bar\Omega)} \leq C'=C(1+\| \phi \|_{C^{1,1}(\partial\Omega)})$,
in particular $h$ is Lipschitz continuous in $\bar\Omega$. Since
$h=\phi$ on $\partial \Omega$ and $A\rho + \hat \phi \leq h$ it implies
$$
	\int_\Omega dd^c h \wedge \beta^{n-1}
		\leq \int_\Omega dd^c (A\rho + \hat \phi)\wedge \beta^{n-1} <+ \infty,
$$
by Lemma~\ref{cp-2}. This verifies \eqref{cor-le1-1}. The proof is finished.
\end{proof}

\begin{lemma}
\label{cor-le-2}
For $0\leq \nu < \frac{1}{2}$, the function $\rho_\nu = - |\rho |^{1-\nu}$,
$\rho$ as in \eqref{rho},
belongs to $SH_m (\Omega)\cap Lip_{1-\nu}(\bar \Omega)$ and satisfies
\begin{equation}
\label{cor-le2-1}
	\int_\Omega d \rho_\nu \wedge d^c \rho_\nu \wedge \beta^{n-1} < +\infty.
\end{equation}
\end{lemma}

\begin{proof}
It follows from formulas
\begin{equation}
\label{cor-le2-2}
	dd^c[-(- \rho)^{1-\nu}]
		= (1-\nu) |\rho |^{-\nu} dd^c\rho
			+ \nu (1-\nu)|\rho |^{-1-\nu} d\rho\wedge d^c\rho,
\end{equation}
and
\begin{equation}
\label{cor-le2-3}
	 d \rho_\nu \wedge d^c \rho_\nu \wedge \beta^{n-1}
		= (1-\nu)^2 |\rho |^{-2\nu} d\rho\wedge d^c\rho \wedge \beta^{n-1}.
\end{equation}
Since $-2\nu>-1$, the integral \eqref{cor-le2-1} converges.
\end{proof}


We are now in the position to prove Theorem~\ref{mcor}. The proof will make use of the envelope
$h$ in Lemma~\ref{cor-le-1}, $\rho_\nu$ in Lemma~\ref{cor-le-2} and $\rho$ from \eqref{rho}.
In what follows, we will use these functions without mentioning them anymore.

\begin{proof}[Proof of $(a)$ in Theorem~\ref{mcor}]
Since $f$ is bounded near the boundary
there is a compact set $F \subset \subset \Omega$ and $M>0$ such that
$0 \leq f \leq M$ in $\Omega \setminus F$.
We may choose $A>0$ big enough such that
$A\rho + h \leq u$ in a neighborhood of $F$,
as $\rho < - \varepsilon$ in $F$ for some $\varepsilon>0$, and
$$
	[dd^c (A\rho +h)]^m \wedge \beta^{n-m}
		\geq (dd^c A\rho )^m \wedge \beta^{n-m}
		\geq f \beta^n
		\;\; \text{ in } \;\; \Omega\setminus F.
$$
Therefore, $A\rho + h \leq u$ in $\Omega \setminus F$ by the comparison principle
(Lemma~\ref{cp-1}).
Therefore,
\begin{equation}
\label{cor-pf-1}
b:= A\rho +h\leq u \mbox{ in } \Omega
\mbox{ and } b \mbox{ is Lipschitz continuous in } \bar\Omega.
\end{equation}
Using Lemma~\ref{cp-2},  \eqref{cor-le1-1} and the fact that
$\rho$ is $C^2$ smooth in a neighborhood of $\bar\Omega$, we get
\begin{equation}
\label{cor-pf-2}
 	\int_\Omega dd^cu \wedge \beta^{n-1}
		\leq \int_\Omega dd^c b \wedge \beta^{n-1} < +\infty.
\end{equation}
According to \eqref{cor-pf-1} and \eqref{cor-pf-2},
the assumptions of Theorem~\ref{main}-(b) are satisfied.
The first part of Theorem~\ref{mcor} follows.
\end{proof}


\begin{proof}[Proof of $(b)$ in Theorem~\ref{mcor}]
From the assumption of $f$ near the boundary,
there is a compact subset $F \subset\subset \Omega$
such that $ f(z) \leq C| \rho|^{-m\nu}$ in $\Omega\setminus F$.
Using \eqref{cor-le2-2} and \eqref{rho}, it follows that
\begin{align*}
 ( dd^c \rho_\nu )^m\wedge \beta^{n-m}
	& \geq (1-\nu)^m (-\rho)^{-m\nu} \sigma \beta^n \\
	& \geq \frac{\sigma(1-\nu)^m}{C} f \beta^n \;\; \text{ in } \;\; \Omega\setminus F.
\end{align*}
Therefore, we may choose $A>0$ so big that
\begin{equation*}
\label{cor-pf2-1}
b_\nu:= A \rho_\nu + h \leq u \mbox{ in a neighborhood of } F,
\end{equation*}
and
\begin{equation*}
\label{cor-pf2-2}
(dd^c b_\nu)^m\wedge \beta^{n-m}
		\geq (dd^c A\rho_\nu)^m \wedge \beta^{n-m}
		\geq f \beta^n \;\; \text{ in }\;\; \Omega \setminus F.
\end{equation*}
Hence, by the comparison principle (Lemma~\ref{cp-1}), we get
\begin{equation*}
\label{cor-pf2-3}
	b_\nu \leq u \mbox{ in } \Omega \setminus F.
\end{equation*}
So,
\begin{equation}
\label{cor-pf2-4}
	b_\nu \leq u \mbox{ in } \Omega \mbox{ and } b_\nu \in Lip_{1-\nu}(\bar\Omega).
\end{equation}
Moreover, Lemma~\ref{cp-2} and \eqref{cor-le2-1} imply
\begin{equation}
\label{cor-pf2-5}
	 \int_\Omega du\wedge d^cu \wedge \beta^{n-1}
 		\leq \int_\Omega d b_\nu \wedge d^c b_\nu \wedge \beta^{n-1} < +\infty.
\end{equation}
According to \eqref{cor-pf2-4} and \eqref{cor-pf2-5},
the assumptions of Theorem~\ref{main}-(a) are satisfied.
Note that $\gamma_2 < \frac{1}{2}  < 1-\nu$.
Thus, the second part in Theorem~\ref{mcor} follows.
\end{proof}

In the last part we consider the homogeneous case of the equation \eqref{heq2},
i.e. the right hand side $f \equiv 0$, when the boundary data is only H\"older continuous.
Similarly to the classical case $m=1$, Proposition~\ref{hh} and the case $m=n$, Theorem 6.2
in \cite{BT1}, it says that any H\"older continuous function on the boundary $\partial\Omega$
can be extended to a  H\"older continuous $m$-subharmonic function in $\Omega$.
For $m=n$, it has been shown in \cite{BT1} that the H\"older exponent is sharp.
\begin{theorem}
 \label{hr}
Let $\Omega$ be a smoothly bounded strongly $m$-pseudoconvex domain
and let $\phi $ belong to $Lip_{2\alpha}(\bar \Omega)$, $0 < \alpha \leq \frac{1}{2}$.
Then the upper envelope
$$
	h_m(z) = \sup \{ v(z): v\in SH_m(\Omega)\cap C(\bar \Omega),
					\;\; v_{|_{\partial \Omega}} \leq \phi\}
$$
is $m$-subharmonic in $\Omega$, and it belongs to $Lip_\alpha(\bar \Omega)$.
 \end{theorem}

\begin{proof}
It is enough to verify the H\"older continuity of $h_m(z)$ since $m$-subharmonicity has been shown in \cite{Bl}.
As in the proof of Proposition~\ref{hh}, one needs
the following $m$-subharmonic barrier at any given point on the boundary.
Let $\rho$ be defined in \eqref{rho}.
\begin{lemma}
\label{hr1}
Suppose that $\|\phi\|_{2\alpha} = M$.
There is a uniform constant
$K = K (\Omega, \rho) >0$ such that for any $\xi \in \partial \Omega$ the function
$$
	b_{ \xi }(z) = - M (|z-\xi |^2  - K \rho )^\alpha + \phi (\xi)
$$
is $m$-subharmonic  in $\Omega$ and belongs to $Lip_\alpha( \bar \Omega)$.
Moreover, it is equal to $\phi (\xi)$ at $\xi$, and $b_\xi (z) \leq \phi (z)$
for every $ z\in \partial \Omega$.
\end{lemma}

\begin{proof}[Proof of Lemma~\ref{hr1}]
We have
\begin{align*}
	dd^c (|z-\xi |^2 - K \rho)^\alpha
	\;&	=  \alpha (|z -\xi |^2 - K \rho(z))^{\alpha -1} dd^c (|z -\xi |^2 - K \rho ) \\
	\;&		- \alpha (1- \alpha) (|z-\xi|^2 - K \rho (z))^{\alpha -2}
				d ( |z-\xi |^2 - K \rho)  \wedge d^c ( |z-\xi|^2 - K \rho ).
\end{align*}
Hence, in $\Omega$  
\begin{align*}
	dd^c b_\xi (z) =
		 	 M \alpha (| z-\xi|^2 -K \rho (z) )^{\alpha -1}
			 	dd^c ( K \rho (z) - |z|^2 )
			+ \Lambda (z , \xi),
\end{align*}
where
\begin{align*}
	\Lambda (z, \xi) = M \alpha (1- \alpha) ( |z- \xi |^2
				 - K \rho (z))^{ \alpha -2} d ( | z-\xi|^2 -K \rho) \wedge d^c (|z-\xi |^2 - K \rho)
\end{align*}
is a positive $(1,1)$ form for any $z \in \Omega$.
Thus,
$$
	dd^c b_\xi (z) =
		\Theta (z , \xi )+ \Lambda (z, \xi),
$$
where $\Theta (z , \xi )= M \alpha (| z-\xi|^2 -K \rho (z) )^{\alpha -1}  dd^c [ K\rho(z) - |z|^2] $.
If we choose $K= K(\Omega, \rho)>0$ big enough, independent of $\xi$, then $K \rho (z) -|z|^2$
is a strictly $m$-sh function in a neighborhood of $\bar\Omega$. It implies that
$\Theta (z, \xi ) \in \Gamma_m$, i.e. the eigenvalues of the matrix of coefficients
$\Theta(z, \xi)$ form  a vector in $\Gamma_m$,
 for every $ z\in \Omega$. Hence, for $1 \leq k \leq m$
$$
	[dd^c b_\xi (z)]^k \wedge \beta^{n-k} (z)
		= [ \Theta (z, \xi) + \Lambda (z, \xi)]^k \wedge \beta^{n-k} (z) > 0
$$
for every $z \in \Omega$.
Therefore $b_\xi (z) \in SH_m (\Omega) \cap Lip_\alpha(\bar \Omega)$ by Definition~\ref{su-de-1}.
The other properties  easily follow from the formula for $b_\xi (z)$.  The proof is finished.
\end{proof}

Set
$$
	b(z) := \sup \{ b_\xi(z) : \; \xi \in \partial \Omega\}.
$$
Since $|b_\xi(z) - b_\xi(w)| \leq C |z-w|^\alpha$, $C = C (\phi, K, \rho)$, one has
$b(z) \in SH_m(\Omega) \cap Lip_\alpha(\bar \Omega)$.
Furthermore, from properties of $b_\xi (z)$ we have $b(z )= \phi (z)$ on $\partial \Omega$.
Hence $ b(z) \leq h_m(z)$ in $\Omega$ by the definition of $h_m(z)$.

Let $h_1(z)$ be the harmonic extension of $\phi$ to $\Omega$.
According to Proposition~\ref{hh} we know that
$h_1(z) \in  Lip_\alpha (\bar\Omega)$, and $ h_m(z) \leq h_1 (z) $ in $\Omega$.
Altogether we have
$$
	b(z) \leq h_m (z) \leq h_1 (z) \mbox { in } \Omega, \;\;\;\;
	b(z) = h_m (z) = h_1 (z) \mbox { on } \partial\Omega.
$$
Having these properties, we may repeat the last part of
the proof of  Proposition~\ref{hh} in order to get
that $h_m (z) \in Lip_\alpha (\bar \Omega)$. The theorem follows.
\end{proof}


\bigskip

\begin{thebibliography}{CNS}

\bibitem[BT1]{BT1}
E. Bedford, B.A. Taylor,
The Dirichlet problem for a  complex Monge-Amp\`ere equation,
\it Invent. Math.\rm\ {\bf 37} (1976), 1-44.

\bibitem[BT2]{BT2}
E. Bedford, B.A. Taylor,
A new capacity for plurisubharmonic functions,
\it Acta Math. \rm {\bf 149} (1982), 1-40.

\bibitem[Bl]{Bl}
Z. B\l ocki,
Weak solutions to the complex Hessian  equation,
\it Ann. Inst. Fourier (Grenoble) \rm\ {\bf 55}, 5 (2005), 1735-1756.


\bibitem[Ce]{Ce1}
U. Cegrell, Pluricomplex energy,
 \it Acta Math. \rm {\bf 180:2} (1998), 187-217.


\bibitem[CDS]{CDS}
X. X. Chen, S. Donaldson, S. Sun,
K\"ahler-Einstein metrics on Fano manifolds, I: approximation of metrics with cone singularities,
\it preprint \rm arXiv: 1211.4566.

\bibitem[Ch1]{Chi1}
L. H. Chinh,
Solutions to degenerate Hessian equations,
\it preprint \rm arXiv: 1202.2436v2.

\bibitem[Ch2]{Chi2}
L. H. Chinh,
Viscosity solution to complex Hessian equations,
\it preprint \rm\ arXiv: 1209.5343v2.

\bibitem[De]{D1}
J.P. Demailly,
Potential theory in several complex variable,
\it Lecture notes. \rm\ ICPAM, Nice, 1989.



\bibitem[DDGKPZ]{Demailly et al}
J.-P. Demailly, S. Dinew, V. Guedj, H. H. Pham, S. Ko\l odziej, A. Zeriahi,
H\"older continuous solutions to Monge-Amp\`ere equations,
\it preprint \rm arXiv: 1112.1388v1, to apear in JEMS.

\bibitem[DK1]{DK1}
S. Dinew, S. Ko\l odziej,
A priori estimates for the complex Hessian equations,
 \it preprint\ \rm\ arXiv: 1112.3063.

\bibitem[DK2]{DK2}
S. Dinew, S. Ko\l odziej,
Liouville and Calabi-Yau type theorems for complex Hessian equations,
\it preprint  \rm\ arXiv: 1203.3995v1.

\bibitem[DNS]{DNS}
T. C. Dinh, V. A. Nguyen and N. Sibony,
Exponential estimates for plurisubharmonic functions and stochastic dynamics,
\it Journal Diff. Geom. \rm {\bf 84} (2010), 465-488.

\bibitem[EGZ]{EGZ}
P. Eyssidieux, V. Guedj, A. Zeriahi,
Singular K\"ahler-Einstein metrics,
\it J. Amer. Math. Soc. \rm\ {\bf 22} (2009), 607-639.

\bibitem[Ga]{Ga}
L. G\aa rding,
An inequality for hyperbolic   polynomials,
\it J. Math. Mech.\rm\  {\bf 8} (1959) 957-965.

\bibitem[GKZ]{GKZ}
V. Guedj, S. Ko\l odziej, A. Zeriahi,
H\"older continuous solutions to Monge-Amp\`ere equations,
 \it Bull. Lond. Math. Soc. \rm\ {\bf 40} (2008), 1070-1080.

\bibitem[GT]{GT}
D. Gilbarg, N. Trudinger, Elliptic partial differential equations of second order,
\it Grundl. der Math. Wiss. Springer Verlarg, \rm\ {\bf 244} (1998).

\bibitem[HL]{HanLin}
Q. Han, F. Lin, Elliptic partial differential equations,
\it Courant Lecture Notes in Mathematics, \rm {\bf 1}, New York  University.


\bibitem[Hou]{Hou}
Z. Hou,
Complex Hessian equation on  K\"ahler manifold,
\it Int. Math. Res. Not. \rm {\bf 16} (2009), 3098-3111.

\bibitem[HMW]{HMW}
Z. Hou, X.-N. Ma, D. Wu,
A second order estimate for complex Hessian equations on a compact K\"ahler manifold,
\it Math. Res. Lett. \rm {\bf 17} (2010), 547-561.

\bibitem[J]{J}
A. Jbilou,
\' Equations hessiennes complexes sur des vari\'et\'es k\"ahl\'eriennes,
\it C. R. Math. Acad. Sci. Paris \rm {\bf 348} (2010), 41-46.

\bibitem[Ko]{Kok}
V. N. Kokarev,
Mixed volume forms and a complex equation of
Monge-Amp\`ere type on K\"ahler manifolds of positive curvature,
\it Izv. RAN. Ser. Mat. \rm {\bf 74:3} (2010), 65-78.


\bibitem[K1]{K1}
S. Ko\l odziej,
The complex Monge-Amp\`ere equation.
\it Acta Math. \rm {\bf 180} (1998), no.1, 69-117.

\bibitem[K2]{K2}
S. Ko\l odziej,
The complex Monge-Amp\`ere equation and pluripotential theory,
\it Memoirs Amer. Math. Soc. \rm\ {\bf 178} (2005) 64p.

\bibitem[K3]{K3}
S. Ko\l odziej,
H\"older continuity of solutions to the complex Monge-Amp\`ere equation with the
right hand side in $L^p$. The case of compact K\"ahler manifolds,
\it  Math. Ann. \rm {\bf 342} (2008), 379-386.

\bibitem[Li]{Li}
S.-Y. Li,
On the Dirichlet problems for symmetric function equations of the eigenvalues of the complex Hessian,
 \it Asian J. Math.\rm\ {\bf 8} (2004), 87-106.

\bibitem[N]{Cuong}
N.-C. Nguyen,
Subsolution theorem for the complex Hessian equation,
\it Preprint \rm\ arXiv: 1212.4627, to apear in Universitatis Iagellonicae Acta Mathematica.

\bibitem[PSS]{PSS}
D. H. Phong, J. Song, J. Sturm,
Complex Monge-Amp\`ere equations,
\it Surveys in Differential Geometry, \rm vol {\bf 17} , 327-411 (2012).


\bibitem[W]{Wa2}
X.-J. Wang,
The $k$-Hessian equation,
\it Lect. Not. Math. \rm {\bf 1977} (2009).

\end{thebibliography}
\end{document}